\documentclass[12pt]{amsart}

\usepackage{amssymb}
\usepackage{hyperref}
\usepackage{amsmath,amscd}
\usepackage[margin=1in]{geometry}

\newtheorem{theorem}{Theorem}[section]

\newtheorem{corollary}[theorem]{Corollary}

\theoremstyle{definition}

\theoremstyle{remark}
\newtheorem{remark}[theorem]{Remark}

\begin{document}

\title{Even and odd compositions with restricted parts} 
\author{Jia Huang}
\address{Department of Mathematics and Statistics, University of Nebraska, Kearney, NE 68849, USA}
\curraddr{}
\email{huangj2@unk.edu}

\begin{abstract}
A result of Legendre asserts that the difference between the numbers of (length) even and odd partitions of $n$ into distinct parts is $0$, $1$, or $-1$; this also follows from Euler's pentagonal number theorem.
We establish an analogous result for compositions and obtain some generalizations that are related to various entries in the On-Line Encyclopedia of Integer Sequences.
\end{abstract}


\maketitle

\section{Introduction}

A \emph{composition} of $n$ is a sequence $\alpha=(\alpha_1,\ldots,\alpha_\ell)$ of positive integers with \emph{size} $|\alpha| := \alpha_1 + \cdots + \alpha_\ell = n$; the \emph{parts} of $\alpha$ are $\alpha_1,\ldots,\alpha_\ell$, and the \emph{length} of $\alpha$ is $\ell(\alpha) := \ell$.
We often drop parentheses and commas when writing a composition whose parts are single digit numbers.
A \emph{partition} of $n$ is a composition of $n$ whose parts are decreasing.
A composition/partition is \emph{even} (or \emph{odd}, resp.) if its length is even (or odd, resp.).
There are $2^{n-1}$ compositions of $n$ since they correspond to binary sequences of length $n-1$.
On the other hand, although there are recursive and asymptotic formulae, no closed formula is known for the number of partitions of $n$.

Compositions and partitions have been extensively studied due to their significance in discrete mathematics, number theory, representation theory, and many other areas.
This paper is motivated by the following result of Legendre~\cite{Legendre}.

\begin{theorem}[Legendre]\label{thm:Legendre}
The number of even partitions of $n$ into distinct parts minus the number of odd partitions of $n$ into distinct parts equals $(-1)^j$ if $n = j(3j\pm 1)/2$ for some integer $j\ge0$ or $0$ otherwise.
\end{theorem}

For example, there is only one even partition ($31$) and only one odd partition ($4$) among all partitions of $n=4$ ($4, 31$, $22$, $211$, $1111$), giving a difference of $0$, and there are exactly two even partitions ($41$, $32$) and one odd partition ($5$) among all partitions of $n=5=2(3j-1)$ ($5, 41, 32, 311, 221, 2111, 11111$), giving a difference of $1=(-1)^j$, where $j=2$.

Although Theorem~\ref{thm:Legendre} is often attributed to Legendre, it can be derived from the following result of Euler, which is known as Euler's pentagonal number theorem since the \emph{pentagonal numbers} are given by $j(3j-1)/2$ for $j=1,2,\ldots$.

\begin{theorem}[Euler]
One has $(1-x)(1-x^2)(1-x^3)\cdots = 1-x-x^2+x^5+x^7-\cdots$, i.e., 
\[ \prod_{n=1}^\infty (1-x^n) = 1 + \sum_{j=1}^\infty (-1)^j \left( x^{j(3j+1)/2} + x^{j(3j-1)/2} \right). \]  
\end{theorem}

Given the subtle differences between partitions and compositions in not only their definitions and enumerative results as mentioned before but also many other aspects (an algebraic manifestation is given by the representation theory of the $0$-Hecke algebra~\cite{H0Tab} compared with the well-known representation theory of the symmetric group), it is natural to ask for an analogous result of Legendre's theorem for compositions.
However, the number of even compositions of $n$ with distinct parts minus the number of odd compositions of $n$ with distinct parts is given by the sequence 
$1, -1, -1, 1, 1, 3, -3, -1, -7, -11, 7, 3, 15, 35, 71, -35, 25, -57, \ldots$, 
which does not resemble Legendre's theorem; see the On-line Encyclopedia of Integer Sequences (OEIS)~\cite[A339435]{OEIS}.
To remedy this, we recall another famous result of Euler. 

\begin{theorem}[Euler]
The number of partitions of $n$ into distinct parts equals the number of partitions of $n$ into odd parts.
\end{theorem}

The number of even partitions of $n$ into odd parts minus the number of odd partitions pf $n$ into odd parts is given by another known sequence $1, -1, 1, -2, 2, -3, 4, -5, 6$, $-8, 10, -12, 15, -18, 22, -27, 32, -38, \ldots$ 
in OEIS~\cite[A081360]{OEIS}.
In fact, this is simply a signed version of the number of partitions of $n$ into odd parts~\cite[A000009]{OEIS} since the length of a partition of $n$ into odd parts has the same parity as $n$. 
Nevertheless, Euler's theorem on partitions of $n$ into distinct/odd parts opens a door for us to think about other related restrictions on the parts of a composition.
In particular, the following composition analogue of Euler's partition theorem comes to our mind.

\begin{theorem}[Cayley---Stanley]\label{thm:comp1}
The number of compositions of $n$ with odd parts equals the number of compositions of $n + 1$ with parts greater than one. 
\end{theorem}

Cayley~\cite{Cayley} showed that the first number in Theorem~\ref{thm:comp1} equals the \emph{Fibonacci number} $F_n$ defined by the recursive relation $F_n:=F_{n-1}+F_{n-2}$ for $n\ge2$ with initial conditions $F_0:=1$ and $F_1:=1$, and the second number in Theorem~\ref{thm:comp1} also equals $F_n$ by Stanley~\cite{EC1}.
Recently, Sills~\cite{Sills} provided a bijective proof for Theorem~\ref{thm:comp1}.
Motivated by Theorem~\ref{thm:comp1}, we provide the following extension of Legendre's theorem to compositions, which involves a periodic sequence $1, 1, 0, -1, -1, 0, \ldots$ of $0$, $1$, and $-1$ with period $6$~\cite[A010892]{OEIS}.

\begin{theorem}\label{thm1}
Define $b_n:=c_{n,o}-c_{n,e}$, where $c_{n,o}$ (or $c_{n,e}$, resp.) is the number of odd (or even, resp.) compositions of $n + 1$ with parts greater than one.
Then $b_n = (-1)^j$ if $n\in\{3j+1, 3j+2\}$ or $b_n=0$ otherwise.
\end{theorem}

We also obtain some generalizations of Theorem~\ref{thm1}, which are outlined below; note that we have ``odd minus even'' instead of ``even minus odd'' to make the difference as simple as possible in our results.

In Section~\ref{sec:Munagi}, we establish a result which includes Theorem~\ref{thm1} as a special case. 
This gives a signed version of an earlier result by Munagi~\cite[Theorem 1.2]{Munagi} (see also Theorem~\ref{thm:comp2}), which generalized Theorem~\ref{thm:comp1} in a similar way as the well-known generalization (Theorem~\ref{thm:Glaisher}) of Euler's partition theorem due to Glaisher.

In Section~\ref{sec:Nyirenda}, we provide a composition analogue of an extension (Theorem~\ref{thm:Nyirenda}) of Legendre's theorem obtained recently by Nyirenda~\cite{Nyirenda} using extra congruence restrictions on the parts of a partition.
A special case (Corollary~\ref{cor:period}) of our result resembles Legendre's Theorem as it involves a periodic sequence whose $j$th term is $(-1)^j$ if $n\in\{3rj+s+1,3rj+r+s+1\}$ for some integer $j\ge0$ or $0$ otherwise, where $r>s\ge0$, and this also implies Theorem~\ref{thm1} when $(r,s)=(1,0)$.

By relaxing the restriction on the parts of a partition, Franklin obtained a further generalization (Theorem~\ref{thm:Franklin}) of Glaisher's theorem.
In our recent work~\cite{CompRes}, we obtained an analogous result for compositions, which includes Munagi's result as a special case.
In Section~\ref{sec:Huang}, we obtain a signed version of this result together with another variation, giving new interpretations for two entries in OEIS~\cite{OEIS}.

Lastly, we ask some questions for future research in Section~\ref{sec:questions} based on various Legendre-type results of Andrews~\cite{Andrews} and Nyirenda---Mugwangwavari~\cite{NM} on partitions with initial repetitions.

\section{A signed version of Munagi's Theorem}\label{sec:Munagi}

First, we recall a well-known result of Glaisher, which specializes to Euler's partition theorem when $k=2$.

\begin{theorem}[Glaisher]\label{thm:Glaisher}
Given an integer $k\ge1$, the number of partitions of $n$ with no part occurring $k$ or more times equals the number of partitions of $n$ with no parts divisible by $k$.
\end{theorem}

Similarly, Munagi~\cite[Theorem 1.2]{Munagi} generalized Theorem~\ref{thm:comp1}, the composition analogue of Euler's theorem, to the following result using the zigzag graphs of compositions.

\begin{theorem}[Munagi]\label{thm:comp2}
The number of compositions of $n$ with parts congruent to $1$ modulo $k$ equals the number of compositions of $n + k - 1$ with parts no less than $k$.
\end{theorem}

Now we provide a signed version of Theorem~\ref{thm:comp2} and prove it in two ways, one using a generating function and the other using a bijection.

\begin{theorem}\label{thm2}
For $k,n\ge1$, let $b_{k,n}:=c_{k,n,o}-c_{k,n,e}$, where $c_{k,n,o}$ (or $c_{k,n,e}$, resp.) is the number of odd (or even, resp.) compositions of $n+k-1$ with parts no less than $k$. 
Then 
\[ b_{k,n} = \sum_{0\le j\le (n-1)/k} (-1)^j \binom{n-1-j(k-1)}{j}. \] 
\end{theorem}

\begin{proof}[Analytic Proof]
We have
\begin{align*}
1-\sum_{n\ge 1} b_{k,n} x^{n+k-1}
&= \sum_{\ell\ge0} \left( -x^k-x^{k+1}-\cdots \right)^\ell 
=\sum_{\ell\ge0} \frac{(-x^k)^\ell}{(1-x)^{\ell}} \\
&=\left(1-\frac{-x^k}{1-x}\right)^{-1} 
=\frac{1-x}{1-x+x^k} \\
& = 1-x^k\sum_{i\ge0}\left(x-x^k\right)^i \\
&= 1-\sum_{i\ge0}x^{k+i} \sum_{j=0}^i \binom{i}{j} (-x^{k-1})^j.
\end{align*}
For $n\ge1$, extracting the coefficient of $x^{n+k-1}$ gives the desired formula for $b_{k,n}$.
\end{proof}

\begin{proof}[Combinatorial Proof]
There is a bijection from compositions of $n+k-1$ with length $j+1$ and parts no less than $k$ to compositions of $n-j(k-1)$ with length $j+1$ by subtracting $k-1$ from each part of a composition.
There are exactly $\binom{n-1-j(k-1)}{j}$ many compositions of $n-j(k-1)$ with length $j+1$ since each of these compositions can be obtained by inserting $j$ bars between $n-j(k-1)$ dots with no two bars adjacent to each other. 
The desired formula for $b_{k,n}$ follows immediately.
\end{proof}

Taking $k=2$ in Theorem~\ref{thm2} gives Theorem~\ref{thm1}, which is a composition analogue of Legendre's theorem.
The sequence $b_{k,n}$ can also be determined by $b_{k,n}=1$ for $n=1,\ldots,k$ and $d_{k,n}=b_{k,n-1}-b_{k,n-k}$ for $n>k$; see the special cases for $k=2$~\cite[A010892]{OEIS}, 
$k=3$~\cite[A050935]{OEIS},
and $k=4$~\cite[A099530]{OEIS}
in OEIS.
Theorem~\ref{thm2} can be viewed as a signed version of Theorem~\ref{thm:comp2} since by either of the above proofs, we can remove $(-1)^j$ in the formula of $b_{k,n}$ given by Theorem~\ref{thm2} and recover a closed formula by Munagi~\cite{Munagi} for the number $a_{k,n}$ of compositions of $n+k-1$ with parts no less than $k$ (the formula of $a_{k,n}$ can also be found in our earlier work~\cite[Eq.~(1)]{CompRes} and for $k=4$, in OEIS~\cite[A003269]{OEIS}).

\section{A further restriction by congruence on parts}\label{sec:Nyirenda}
In this section, we generalize Theorem~\ref{thm2} by further imposing a congruence condition on the already restricted parts of the compositions.
This is in the spirit of the following extension (slightly rephrased) of Legendre's theorem by Nyirenda~\cite{Nyirenda}.

\begin{theorem}[Nyirenda]\label{thm:Nyirenda}
Let $d_e(n,r)$ (or $d_o(n,r)$, resp.) denote the number of partitions of $n$ into an even (or odd, resp.) number of distinct parts, all of which are congruent to $0$ or $2r\pm1$ modulo $4r$.
Then 
\[ d_e(n,r)-d_o(n,r) = 
\begin{cases}
(-1)^j & \text{if $n=j(2rj\pm1)$ for some integer $j\ge0$}; \\
0 & \text{otherwise}. 
\end{cases} \]
Let $c_e(n,r)$ (or $c_o(n,r)$, resp.) denote the number of partitions of $n$ into an even (or odd, resp.) number of distinct parts, all of which are congruent to $0$ or $\pm r$ modulo $2r+1$.
Then 
\[ c_e(n,r)-c_o(n,r) = 
\begin{cases}
(-1)^j & \text{if $n=j((2r+1)j\pm1))/2$ for some integer $j\ge0$}; \\
0 & \text{otherwise}. 
\end{cases} \]
\end{theorem}

We provide an analogue of Theorem~\ref{thm:Nyirenda} for compositions with two proofs.

\begin{theorem}\label{thm3}
Given integers $k,n\ge1$ and $r>s\ge 0$, let $b_{k,n}^{r,s}:=c_{k,n,o}^{r,s}-c_{k,n,e}^{r,s}$, where $c_{k,n,o}^{r,s}$ (or $c_{k,n,e}^{r,s}$, resp.) is the number of odd (or even, resp.) compositions of $n+k-1$ with parts no less than $k$ and congruent to $k+s$ modulo $r$.
Then
\[ b_{k,n}^{r,s} =
\displaystyle \sum_{ri+j(k+s)=n-1-s} (-1)^j \binom{i+j}{i}. 
\] 
\end{theorem}

\begin{proof}[Analytic Proof]
We have
\begin{align*}
1 - \sum_{n\ge1} b_{k,n}^{r,s} x^{n+k-1}
&= \sum_{\ell\ge0} \left( -x^{k+s}-x^{k+r+s}-x^{k+2r+s}+ \cdots \right)^\ell \\
& =\sum_{\ell\ge0} \frac{(-x^{k+s})^\ell}{(1-x^r)^{\ell}} 
=\left(1-\frac{-x^{k+s}}{1-x^r}\right)^{-1} \\
&=\frac{1-x^r}{1-x^r+x^{k+s}} 
= 1-x^{k+s}\sum_{i\ge0}\left(x^r-x^{k+s}\right)^i \\
&= 
1-\sum_{i\ge0}x^{k+ri+s} \sum_{j=0}^i (-1)^j \binom{i}{j} x^{j(k+s-r)}. 
\end{align*}
Extracting the coefficient of $x^{n+k-1}$ and replacing $i$ with $i+j$ gives the desired formula for $b_{k,n}^{r,s}$.
\end{proof}

\begin{remark}\label{rem}
The above proof is valid even though the exponent of $x^{j(k+s-r)}$ could be zero (when $k=r-s$) or negative (when $k<r-s$).
Alternatively, one can deal with the cases $k=r-s$ and $k>r-s$ separately using similar techniques and obtain the same formula for $b_{k,n}^{r,s}$.
\end{remark}

\begin{proof}[Combinatorial Proof]

The allowed parts (no less than $k$ and congruent to $k+s$ modulo $r$) are $k+s, k+r+s, k+2r+s, \ldots$.
Dividing each part minus $k+s-r$ by $r$ gives a bijection from compositions of $n+k-1$ with exactly $j+1$ parts, each less than $k$ and congruent to $k+s$, to compositions of $i+j+1$ of length $j+1$, where $r(i+j+1)+(k+s-r)(j+1) = n+k-1$, i.e., $ri+j(k+s)=n-1-s$.
The number of compositions of $i+j+1$ of length $j+1$ is $\binom{i+j}{j}$.
Therefore the desired formula for $b_{k,n}^{r,s}$ holds.
\end{proof}

Theorem~\ref{thm3} recovers Theorem~\ref{thm2} when $(r,s)=(1,0)$. 
The following is another special case as mentioned in Remark~\ref{rem}.

\begin{corollary}
Suppose $r>s\ge0$ and $k=r-s$.
Then $b_{k,n}^{r,s}=1$ if $n=s+1$ and $b_{k,n}^{r,s}=0$ if $n\ne s+1$. 
\end{corollary}

\begin{proof}
This can be derived from the formula of $b_{k,n}^{r,s}$ in Theorem~\ref{thm3} or by using the generating function 
\[ \frac{1-x^r}{1-x^r+x^{k+s}} = 1-x^r \]
in the proof of Theorem~\ref{thm3} when $k=r-s$.
\end{proof}

We give one more corollary of Theorem~\ref{thm3} below, which resembles Legendre's theorem and recovers Theorem~\ref{thm1} when $(k,r,s)=(2,1,0)$. 

\begin{corollary}\label{cor:period}
Given integers $r>s\ge0$ and $k=2r-s$, we have $d_{k,n}^{r,s} = (-1)^j$ if $n\in\{3rj+s+1,3rj+r+s+1\}$ for some integer $j\ge0$ or $d_{k,n}^{r,s} =0$ otherwise.
\end{corollary}

\begin{proof}
Suppose $k=2r-s$. 
It follows from the proof of Theorem~\ref{thm3} that
\begin{align*}
1-\sum_{n\ge1} b_{k,n}^{r,s} x^{n+k-1}
&= \frac{1-x^r}{1-x^r+x^{2r}} 
= \frac{1-x^{2r}}{1+x^{3r}} \\
& = \sum_{i\ge0} (-1)^i x^{3ri} - \sum_{j\ge0} (-1)^j x^{2r+3rj}. 
\end{align*}
Extracting the coefficient of $x^{n+k-1}$ after replacing $i$ with $j+1$ for $i\ge1$ gives the desired formula for $d_{k,n}^{r,s}$.
\end{proof}

By Corollary~\ref{cor:period}, if $k=2r-s$ then the sequence $(d_{k,n}^{r,s}: n\ge1)$ has period $6r$ and, upon a backward shift of $r$ terms, its generating function becomes $(1-x^r+x^{2r})^{-1}$.
This generating function is the inverse of the $6r$th cyclotomic polynomial at least when $k+s=2r=4, 6, 8$; see OEIS~\cite[A014021, A014027, A014033]{OEIS}.

\section{A relaxation for restricted parts}\label{sec:Huang}

The following result of Franklin recovers Glaisher's theorem (Theorem~\ref{thm:Glaisher}) when $m=0$.

\begin{theorem}[Franklin]\label{thm:Franklin}
Given integers $k\ge1$ and $m\ge0$, the number of partitions of $n$ with $m$ distinct parts each occurring $k$ or more times equals the number of partitions of $n$ with exactly $m$ distinct parts divisible by $k$.
\end{theorem}

In our recent work~\cite{CompRes}, we obtained a composition analogue of Franklin's theorem.

\begin{theorem}[Huang]\label{thm:comp3}
For any integers $k\ge1$ and $m\ge0$, the number of compositions of $n$ with exactly $m$ parts not congruent to $1$ modulo $k$, each of which is greater than $k$, equals the number of compositions of $n+k-1$ with exactly $m$ parts less than $k$, each of which is preceded by a part at least $k$ and followed by either the last part or a part greater than $k$.
\end{theorem}

Theorem~\ref{thm:comp3} recovers Munagi's theorem when $m=0$.
Our proof for Theorem~\ref{thm:comp3} was based on the bijective proof of Theorem~\ref{thm:comp1} by Sills~\cite{Sills}.
We also established two closed formulae~\cite[Theorem~1.7]{CompRes} for the two equal numbers in Theorem~\ref{thm:comp3}:
\begin{align*}
a_{k,n}^{(m)} 
&= \sum_{ \substack{ \lambda\subseteq (k-2)^m \\ i+(k+1)m+jk+ |\lambda| = n }} \binom{i}{m} \binom{i+j-1}{j} m_\lambda(1^m) \\
&= \sum_{i+(k+1)m+jk+\ell(k-1)+h=n} (-1)^\ell \binom{i}{m} \binom{i+j-1}{j} \binom{m}{\ell} \binom{m+h-1}{h}.
\end{align*}
Here $\lambda\subseteq r^d$ means $\lambda$ is a partition with no more than $d$ parts, each at most $r$,
and $m_\lambda(1^d)$ is the specialization of the monomial symmetric function indexed by the partition $\lambda$ evaluated at the vector $(\underbrace{1,\ldots,1}_d)$, i.e., with $m_i$ denoting the number of parts of the partition $\lambda\subseteq r^d$ that are equal to $i$ for $i=0,1,\ldots,r$, 
\[ m_\lambda(1^d) = \binom{m}{m_0,\ldots,m_r} = \frac{m!}{m_0!\cdots m_r!}. \]

Now we provide a signed version of the above formulae of the number $a_{k,n}^{(m)}$.

\begin{theorem}\label{thm4}
Let $b_{k,n}^{(m)}:=c_{k,n,o}^{(m)}-c_{k,n,e}^{(m)}$, where $c_{k,n,o}^{(m)}$ (or $c_{k,n,e}^{(m)}$, resp.) is the numbers of odd (or even, resp.) compositions of $n+k-1$ with exactly $m$ parts less than $k$, each of which is preceded by a part at least $k$ and followed by either the last part or a part greater than $k$.
Then
\begin{align*}
b_{k,n}^{(m)} 
&= \sum_{ \substack{ \lambda\subseteq (k-2)^m \\ i+(k+1)m+jk+ |\lambda| = n }} (-1)^{j} \binom{i}{m} \binom{i+j-1}{j} m_\lambda(1^m) \\
&= \sum_{i+(k+1)m+jk+\ell(k-1)+h=n} (-1)^{\ell+j} \binom{i}{m} \binom{i+j-1}{j} \binom{m}{\ell} \binom{m+h-1}{h}.
\end{align*}
\end{theorem}

\begin{proof}
Both the analytic and combinatorial proofs of the above formulae of $a_{k,n}^{(m)}$ given in our previous work~\cite[Theorem~1.7]{CompRes} were based on the interpretation of $a_{k,n}^{(m)}$ as the number of the first kind of compositions in Theorem~\ref{thm:comp3}.
However, we can apply the bijective proof of Theorem~\ref{thm:comp3}~\cite[Theorem~1.6]{CompRes} to the combinatorial proof of the above formulae of $a_{k,n}^{(m)}$ and obtain that the length of each of the second kind of compositions in Theorem~\ref{thm:comp3} is given by $2m+j+1$, where $m$ and $j$ are as in the above formulae of $a_{k,n}^{(m)}$.
Thus we have the desired formulae for $b_{k,n}^{(m)}$.
\end{proof}

Theorem~\ref{thm4} provides a new interpretation for a known sequence~\cite[A281862]{OEIS}, which coincides with $b_{k,n}^{(m)}$ when $k=m=2$.
We also have a variation of $b_{k,n}^{(m)}$ with simplified restrictions on the parts of a composition.

\begin{theorem}\label{thm4'}
Let $\bar b_{k,n}^{(m)} := \bar c_{k,n,o}^{(m)} - \bar c_{k,n,e}^{(m)}$, where $\bar c_{k,n,o}^{(m)}$ (or $\bar c_{k,n,e}^{(m)}$, resp.) is the number of odd (or even, resp.) compositions of $n+k-1$ with exactly $m$ parts less than $k$.
Then
\[ \bar b_{k,n}^{(m)} 
= \sum_{i+j+(k-1)(\ell+i-m-1)=n} (-1)^{i+\ell+1} \binom{i+j-1}{j} \binom{i}{m} \binom{m}{\ell}
\]
\end{theorem}

\begin{proof}
We have
\begin{align*}
-\sum_{n\ge 1-k} \bar b_{k,n}^{(m)} x^{n+k-1}y^m
&= \sum_{i\ge0} \left(-y(x+x^2+\cdots+x^{k-1}) -x^k-x^{k+1}-\cdots \right)^i \\
&= \sum_{i\ge0} \left( \frac{-xy(1-x^{k-1})}{1-x} + \frac{-x^k}{1-x} \right)^i \\
&= \sum_{i\ge0} \left( \frac{-xy+x^{k}y-x^k}{1-x} \right)^i
\\
&= \sum_{i\ge0} (-x)^i \sum_{j\ge0} \binom{i+j-1}{j} x^j \sum_{m=0}^i \binom{i}{m} y^m(1-x^{k-1})^m x^{(k-1)(i-m)} \\
&= \sum_{i,j\ge0} (-x)^i \binom{i+j-1}{j} x^j \sum_{m=0}^i \binom{i}{m} y^m \sum_{\ell=0}^m \binom{m}{\ell} (-x^{k-1})^{\ell} x^{(k-1)(i-m)}.
\end{align*}
Extracting the coefficient of $x^{n+k-1}y^m$ gives the desired formula for $\bar b_{k,n}^{(m)}$.
\end{proof}

We find a sequence in OEIS~\cite[A122918]{OEIS} that coincides with $(-1)^n \bar b_{k,n}^{(m)}$ when $m=2$ and $k=1$.

\section{Questions}\label{sec:questions}

Andrews~\cite{Andrews} defined a partition of $n$ to have \emph{initial $k$-repetitions} if every part less than $j$ is repeated at least $k$ times whenever a part $j$ is repeated at least $k$ times; taking $k=1$ in this definition gives \emph{partitions without gaps}, which were first studied by Fine~\cite{Fine}.
Andrews~\cite{Andrews} established the following results on partitions with initial $k$-repetitions. 
\begin{itemize}
\item
The number of partitions of $n$ with initial $k$-repetitions equals the number of partitions of $n$ into parts indivisible by $2k$ and by Glaisher's theorem, also equals the number of partitions of $n$ with no parts occurring $2k$ or more times.
\item
Let $D_e(m,n)$ (or $D_o(m,n)$, resp.) denotes the number of partitions of $n$ with initial $2$-repetitions and with $m$ different parts, of which an even (or odd, resp.) number have multiplicity one. 
Then
\[ 
D_e(m,n)-D_o(m,n) = 
\begin{cases}
(-1)^j, & \text{if } m=j,\ n=j(j+1)/2,\ j\ge0; \\
0, & \text{otherwise}.
\end{cases} 
\]
\end{itemize}
Since the last result resembles Legendre's theorem, we ask for a composition analogue, which may require an appropriate definition of compositions with ``initial $k$-repetitions.''
If this could be done, it would also be interesting to search for composition analogues of various Legendre-type theorems obtained recently by Nyirenda and Mugwangwavari~\cite{NM} based on work of Andrews~\cite{Andrews}.

\section*{Acknowledgment}

The author uses SageMath to help verify the closed formulae obtained in this paper.

\end{document}